\documentclass[12pt,a4paper]{article}
\usepackage{stmaryrd}
\usepackage{booktabs}
\usepackage{subfigure}
\usepackage{epsf,epsfig,amsfonts,amsgen,amsmath,amstext,amsbsy,amsopn,amsthm}
\usepackage{amsmath,times,mathptmx}

\usepackage{amsfonts,amsthm,amssymb,mathrsfs,bbding}
\usepackage{indentfirst}
\usepackage{graphicx}
\usepackage{graphics,multicol}
\usepackage{color}
\usepackage[colorlinks=true,anchorcolor=blue,filecolor=blue,linkcolor=blue,urlcolor=blue,citecolor=blue]{hyperref}
\usepackage{floatrow}
\usepackage{enumerate}
\usepackage{tikz}
\usepgflibrary{bbox}

\usetikzlibrary{calc}
\evensidemargin=\oddsidemargin\topmargin=-1.5cm

\newtheorem{thm}{Theorem}[section]
\newtheorem{prob}{Problem}[section]

\newtheorem{lem}{Lemma}[section]

\newtheorem{conj}{Conjecture}[section]

\newtheorem{claim}{Claim}[section]

\makeatletter
\@addtoreset{case}{lem}
\makeatother

\makeatletter
\@addtoreset{claim}{section}
\makeatother

\allowdisplaybreaks\allowdisplaybreaks[4]

\addtocounter{section}{0}

\begin{document}
\title{Two classes of connectivity-related non-Hamiltonian 1-planar perfect  graphs \footnote{Supported by the National Natural Science Foundation of China   (Grant No. 12271157),
the Natural Science Foundation of Hunan Province, China (Grant No. 2022JJ30028, 2023JJ30178), Natural Science Foundation of Changsha (No. kq2208001),  and Hunan Provincial Department of Education(No. 21A0590)}}

\author{ 
{\bf Licheng Zhang$^{a}$}\thanks{
lczhangmath@163.com}, {\bf Shengxiang Lv$^{b}$}\thanks {
lvsxx23@126.com}, {\bf Yuanqiu Huang $^{c}$}\thanks{Corresponding author:
hyqq@hunnu.edu.cn}
\\
\small $^{a}$ School of Mathematics, Hunan University, 
\small Changsha, 410082, China\\
\small $^{b}$ School of Mathematics and Statistics, Hunan University of Finance and Economics,\\
\small Changsha, 410205, China\\
\small $^{c}$College of Mathematics and Statistics, Hunan Normal University, 
\small Changsha, 410081, China 
}

\date{}
\maketitle {\flushleft\large \textbf{Abstract}}
The existence of Hamiltonian cycles in 1-planar graphs with higher connectivity has attracted considerable attention. Recently, the authors and Dong proved that 4-connected 1-planar chordal graphs are Hamiltonian-connected. In this paper, we investigate the non-Hamiltonicity of a broader class of graphs, specifically perfect graphs, under the constraint of 1-planarity, with a focus on connectivity of at most 5. We also propose some unsolved problems.

\begin{flushleft}
\end{flushleft}

\section{Introduction}
This paper considers only finite simple graphs.  For terminology and notation, we follow the books \cite{Bondy} and \cite{West}. For a graph $G$, let $V(G)$ and $E(G)$ denote its vertex set and edge set, respectively. A subset $S$ of $V(G)$ is called a \emph{separator} if $G-S$ is disconnected. For a non-complete graph $G$, its \emph{connectivity}, denoted by $\kappa(G)$, is defined to be the minimum size of any separator $S$ of $G$. If  $G$ is complete, then $\kappa(G)=|V(G)|-1$. A graph is \emph{$k$-connected} if its connectivity is at least $k$.
A \emph{Hamiltonian  cycle }in a graph $G$ is a  cycle that visits every vertex of
$G$ exactly once. A graph $G$ is  \emph{Hamiltonian }if  it contains a Hamiltonian cycle.

There has been a long-standing interest in finding Hamiltonian cycles in graphs embedded on surfaces.
 In planar graphs, Hamiltonicity is closely tied to connectivity, as demonstrated by Tutte's theorem, which asserts that every $4$-connected planar graph is Hamiltonian \cite{Tutte}.
Following this, researchers have established analogous Hamiltonicity results for graphs on other surfaces, such as $4$-connected projective-planar graphs \cite{Thomas} and $5$-connected toroidal graphs \cite{Kawarabayashi}.  For graphs with a small crossing number, Ozeki and  Zamfirescu \cite{Ozeki} proved that every 4-connected graph with a crossing number at most 2 is
Hamiltonian, and they constructed  4-connected non-Hamiltonian graphs with crossing
number at least 6.

\emph{$1$-planar graphs} serve as a natural generalization of planar graphs; they can be drawn in the plane such that each edge crosses at most once. This class of graphs has been extensively studied (see \cite{Barat, Biedl2022, Suzuki2010} for more details). 
In recent years, researchers have sought to establish foundational results regarding the Hamiltonicity of $1$-planar graphs with high connectivity.  
In 2012, Hud\'{a}k, Madaras, and Suzuki \cite{Hudak} proved that every $7$-connected maximal $1$-planar graph is Hamiltonian. In 2019, Biedl \cite{Biedl2019} improved upon this result by showing that every $4$-connected maximal $1$-planar graph is Hamiltonian. Subsequently, in 2020,  Fabrici et al. \cite{Fabrici} showed that a $3$-connected locally maximal $1$-planar graph $G$ is Hamiltonian if it has at most three $3$-vertex-cuts, which implies that all $4$-connected maximal $1$-planar graphs are Hamiltonian.  Despite these advances, there are still many non-Hamiltonian $1$-planar graphs with connectivity $4$ or $5$ \cite{Biedl2019,Fabrici}.  It is known that every $1$-planar graph with  $n\ge 3$
vertices has at most  $4n-8$ edges \cite{Bodendiek}, 
 and its minimum degree is at most $7$, which consequently implies its connectivity is also at most $7$. In fact, it remains an open question whether $1$-planar graphs with connectivity $6$ or $7$ are Hamiltonian.

\begin{prob}[\cite{Biedl2019,Fabrici}]\label{p1}
Is every 6-connected (or 7-connected) 1-planar graph  Hamiltonian?
\end{prob}

A \emph{hole} is  an induced cycle of length at least four, while a graph is \emph{chordal} if it contains no holes. Very recently, we (the authors) and Dong proved the following theorem.

\begin{thm}[\cite{Zhang}]
Every $4$-connected  $1$-planar chordal graph are Hamiltonian-connected.
\end{thm}   

 A graph $G$ is said to be \emph{perfect} if, for every induced subgraph of $G$, the clique number is equal to the chromatic number.  It is well known that chordal graphs are perfect \cite{Hougardy}.  This leads to the natural question of \emph{whether $4$-connected perfect graphs are either Hamiltonian or Hamiltonian-connected.} In this paper, we prove, by constructing two classes of 1-planar perfect graphs, that connectivity of $4$ or even $5$ is not sufficient to guarantee that these graphs are Hamiltonian.

 \begin{thm}\label{main1}
For any positive integer $k\le 5$, there exist infinitely many  non-Hamiltonian $1$-planar perfect graphs with connectivity $k$.
\end{thm}

It is known that bipartite graphs are perfect. The non-Hamiltonian 1-planar perfect graphs with connectivity 5 presented in this paper are bipartite, setting them apart from those in the existing literature, which are neither bipartite nor perfect.
 
 Before proving the main theorem, we will introduce some definitions and notations. A planar drawing divides the plane into topologically connected regions called \emph{faces}. A $1$-plane graph also divides the plane into faces, but the boundaries of these faces may consist of edge segments between vertices and/or crossing points of edges. If the boundary of a face $f$  in a $1$-plane drawing 
consists solely of vertices (with no crossing points), then 
$f$ is  referred to as an \emph{uncrossed face},  otherwise, it is called a \emph{crossed face}.
 
%
Given a graph $G=(V(G), E(G))$ and two subsets $X, Y\subset V(G)$ such that $X\cap Y=\emptyset$, we denote $E_G(X, Y)$ as the set of edges in $E(G)$ that have one end in $X$ and the other end in $Y$.
 When there is only one vertex $x$ in $X$, we denote $E_G(x, Y)$ as the set of edges in the graph 
$G$ that connect the vertex $x$ to the vertices in the set $Y$. 
We use $\omega(G)$ to denote the clique number of a graph and $\chi(G)$ to represent its chromatic number. We denote by $N_G(x)$ the neighborhood  of a vertex
$x$ in a graph $G$.

\section{Proof of Theorem \ref{main1}}

To obtain our main results, we will first present the following lemmas.

\begin{lem}[\cite{Hougardy}]\label{lem-1}
(i) ~Every bipartite graph is perfect; ~(ii) ~Every chordal graph is perfect.
\end{lem}


\begin{lem}[page 287, \cite{West}]\label{lem-3}
Let $G=(X,Y)$ be a bipartite graph. If $G$ is  Hamiltonian, then $|X|=|Y|$, which implies that $|V(G)|$ is even.
\end{lem}

\begin{lem}[page 213, \cite{Bondy}]\label{lem-4}
Let $G$ be a $k$-connected graph, and let $H$ be a graph obtained from $G$ by adding a new vertex $y$ and connecting it to at least $k$ vertices of $G$. Then $H$ is also $k$-connected.
\end{lem}

A \emph{quadrangulation} is a plane graph in which all faces are bounded by a cycle of length $4$.
Let $Q_s$ be the graph obtained by \emph{double-stellating }every face of a quadrangulation $Q$, which involves inserting  two new vertices into each face of $Q$ that are adjacent to all vertices of that face.
 We refer to $Q_s$  as a \emph{double-stellating quadrangulation}.  An example is shown in Fig. \ref{f1}. The newly inserted vertices are called \emph{stellating  vertices} in $Q_s$.  The term double-stellating is used in accordance with  Biedl \cite{Biedl2019}.

 \begin{figure}
 \centering 
 \begin{tikzpicture}[ scale=0.4, bezier bounding box,
    whitenode/.style={circle, draw=black, fill=red, thick,minimum size=0.5mm, inner sep =3pt},  
    bluenode/.style={circle, draw=black, fill=blue, thick,minimum size=0.5mm, inner sep =3pt},  
    springnode/.style={circle, draw=black, fill=white, minimum size=0.5mm, inner sep =2pt},  
    black_thick/.style={line width=2pt},  
    blackedge/.style={line width=0.5pt}  
]

 \begin{scope}[xshift=-12cm,yshift=-1cm]
	\node [style=whitenode] (0) at (-5, 7) {};
		\node [style=bluenode] (1) at (5, 7) {};
		\node [style=bluenode] (2) at (-5, -3) {};
		\node [style=whitenode] (3) at (5, -3) {};
		\node [style=bluenode] (4) at (-2, 4) {};
		\node [style=whitenode] (5) at (-2, 0) {};
		\node [style=whitenode] (6) at (2, 4) {};
		\node [style=bluenode] (7) at (2, 0) {};

		\draw [style={black_thick}] (1) to (0);
		\draw [style={black_thick}] (0) to (2);
		\draw [style={black_thick}] (2) to (3);
		\draw [style={black_thick}] (1) to (3);
		\draw [style={black_thick}] (0) to (4);
		\draw [style={black_thick}] (4) to (6);
		\draw [style={black_thick}] (6) to (7);
		\draw [style={black_thick}] (7) to (5);
		\draw [style={black_thick}] (5) to (4);
		\draw [style={black_thick}] (5) to (2);
		\draw [style={black_thick}] (7) to (3);
		\draw [style={black_thick}] (6) to (1);

 \end{scope}

 \begin{scope}[xshift=5cm]

      	\node [style=whitenode] (0) at (-5, 6) {};
		\node [style=bluenode] (1) at (5, 6) {};
		\node [style=bluenode] (2) at (-5, -4) {};
		\node [style=whitenode] (3) at (5, -4) {};
		\node [style=bluenode] (4) at (-2, 3) {};
		\node [style=whitenode] (5) at (-2, -1) {};
		\node [style=whitenode] (6) at (2, 3) {};
		\node [style=bluenode] (7) at (2, -1) {};
		\node [style=springnode] (8) at (-2, 4.5) {};
		\node [style=springnode] (9) at (2, 4.5) {};
		\node [style=springnode] (10) at (-3.5, 3) {};
		\node [style=springnode] (11) at (-3.5, -1) {};
		\node [style=springnode] (12) at (-0.75, 1) {};
		\node [style=springnode] (13) at (0.75, 1) {};
		\node [style=springnode] (14) at (3.5, 3) {};
		\node [style=springnode] (15) at (3.5, -1) {};
		\node [style=springnode] (16) at (-1.25, -2.75) {};
		\node [style=springnode] (17) at (1.5, -2.75) {};
		\node [style=springnode] (18) at (0, 8) {};
		\node [style=springnode] (19) at (0, -5.75) {};
 
        \draw [black_thick] (1) -- (0) -- (2) -- (3) -- (1);  
        \draw [black_thick] (0) -- (4) -- (6) -- (7) -- (5) -- (4);  
        \draw [black_thick] (5) -- (2) (7) -- (3) (6) -- (1);  
          
        \draw [blackedge] (0) -- (10) -- (4);  
        \draw [blackedge] (10) -- (2);  
        \draw [blackedge] (11) -- (4);  
        \draw [blackedge] (11) -- (2);  
        \draw [blackedge] (11) -- (0);  
        \draw [blackedge] (11) -- (5);  
        \draw [blackedge] (10) -- (5);  
        \draw [blackedge] (4) -- (12);  
        \draw [blackedge] (12) -- (5);  
        \draw [blackedge] (12) -- (7);  
        \draw [blackedge] (12) -- (6);  
        \draw [blackedge] (13) -- (4);  
        \draw [blackedge] (13) -- (5);  
        \draw [blackedge] (13) -- (6);  
        \draw [blackedge] (13) -- (7);  
        \draw [blackedge] (6) -- (14);  
        \draw [blackedge] (14) -- (1);  
        \draw [blackedge] (14) -- (3);  
        \draw [blackedge] (14) -- (7);  
        \draw [blackedge] (6) -- (15);  
        \draw [blackedge] (15) -- (3);  
        \draw [blackedge] (15) -- (7);  
        \draw [blackedge] (15) -- (1);  
        \draw [blackedge] (0) -- (8);  
        \draw [blackedge] (8) -- (4);  
        \draw [blackedge] (8) -- (6);  
        \draw [blackedge] (8) -- (1);  
        \draw [blackedge] (0) -- (9);  
        \draw [blackedge] (9) -- (4);  
        \draw [blackedge] (9) -- (6);  
        \draw [blackedge] (9) -- (1);  
        \draw [blackedge] (18) -- (0);  
        \draw [blackedge] (18) -- (1);  
        \draw [blackedge] (5) -- (16);  
        \draw [blackedge] (16) -- (2);  
        \draw [blackedge] (16) -- (7);  
        \draw [blackedge] (5) -- (17);  
        \draw [blackedge] (17) -- (3);  
        \draw [blackedge] (7) -- (17);  
        \draw [blackedge] (17) -- (2);  
        \draw [blackedge] (3) -- (16);  
        \draw [blackedge] (2) -- (19);  
        \draw [blackedge] (19) -- (3);  
		\draw [style=blackedge, in=-135, out=180, looseness=1.50] (19) to (0);
		\draw [style=blackedge, in=-45, out=0, looseness=1.50] (19) to (1);
		\draw [style=blackedge, in=135, out=-180, looseness=1.50] (18) to (2);
		\draw [style=blackedge, in=45, out=0, looseness=1.50] (18) to (3);
         \end{scope}
\end{tikzpicture}  
\caption{On the left is a quadrangulation $Q$, and on the right is the double-stellating quadrangulation $Q_s$.}
\label{f1}
  \end{figure}
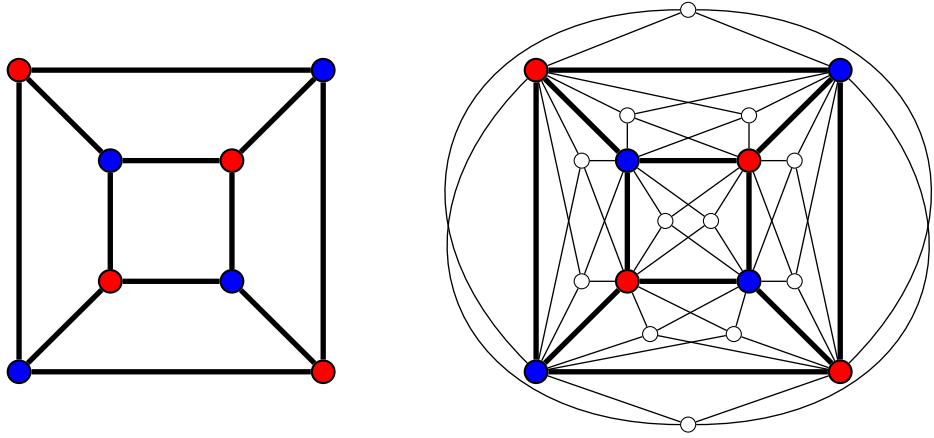

\begin{lem}\label{lem-5}
Let $Q_s$ be a double-stellating quadrangulation. Then, $Q_s$ is perfect.
\end{lem}
\begin{proof}
Let $S$ be the set of all stellating vertices of $Q_s$. It is clear that $S$ forms an independent set  in $Q_s$. Since $Q$ is bipartite, it follows that $Q_s$  is tripartite, implying $\chi(Q_s)\le 3$. Furthermore, since $Q_s$ contains a triangle, we have  $\omega(Q_s)=3$,  thus $\chi(Q_s)=3$. 

Let $H$ be an induced subgraph of $Q_s$. By the definition of perfect graphs, it suffices to show that $\omega(H)=\chi(H)$.
\begin{itemize}
  \item[(1).] If $V(H)=\emptyset$, then $\omega(H)=\chi(H)=0$.
  \item[ (2).]If $V(H)\ne \emptyset$ and $E(H)= \emptyset$,  then $\omega(H)=\chi(H)=1$. 
  \item[ (3).] If $V(H)\ne \emptyset$ and $E(H)\ne \emptyset$, then $2 \le \omega(H) \le \chi(H) \le \chi(Q_s) =3$. 
\end{itemize}

We can 3-color $ H$ by first 2-coloring $ Q $ with red and blue, and then letting $ V(H)\cap V(Q) $ inherit the 2-coloring from $Q$. We color stellating vertices $w$ in $H$ in the following order:
\begin{itemize}
\item[(step 1)]if all vertices in $N_H(w)$ are of  a single color (e.g., red), we assign $ w$ the opposite color (blue). 
\item[(step 2)]if $N_H(w)$ contains both red and blue vertices, we assign $ w $ a third color.
\end{itemize}

  If all stellating vertices follow the first step, $H$ is bipartite, leading to $\omega(H) =\chi(H) =2$. 
  Otherwise, there exists a stellating vertex $u$ with two neighbors $x$ and $y$ of opposite colors in $H$. These vertices $x$ and $y$ are consecutive vertices of the face in $Q$ that $u$ and another stellating  vertex have double-stellated. Thus, $u,x,$ and $y$ induce  a triangle in $H$. In this case, $\omega(H) = \chi(H)=3$.

Therefore, $Q_s$ is perfect, as desired.
\end{proof}

\begin{proof}[\textbf{Proof of Theorem~\ref{main1}}]

Clearly,  $k=1$ is trivial since graphs with connectivity one are all non-Hamiltonian. For $k=2$, we consider bipartite graphs 
$K_{2,n}$
  where $n \ge 3$. These graphs are planar (and also $1$-planar), perfect by Lemma \ref{lem-1}(i), and non-Hamiltonian according to Lemma \ref{lem-3}. For $k=3$, the authors in \cite{Zhang} constructed non-Hamiltonian $1$-planar chordal graphs with connectivity $3$, which are perfect by Lemma \ref{lem-1}(ii). For $k=4$, double-stellating quadrangulations meet the requirements. By Lemma \ref{lem-5}, they are perfect and clearly $1$-planar graphs. Their non-Hamiltonicity and connectivity of $4$ are proved in Theorem 1 of \cite{Biedl2019}.  For $k = 5$, we shall construct a class of graphs $G_k$.Before that, we need to define the subgraphs $H_k$ of $G_k$ and provide some properties of the subgraphs $H_k$.

Let $k \geq 10$ and $k$ be even. We begin by defining $4$ distinct cycles 
\begin{eqnarray*}
C_1=b_1b_2b_3\cdots b_k, & &C_2=u_1u_2u_3\cdots u_{3k}, \\
C_3=v_1v_2v_3\cdots v_{3k}, &  & C_4=a_1a_2a_3\cdots a_{k}.
\end{eqnarray*}
 and the following edge sets
 \begin{eqnarray*}
E(C_1,C_2)&=&\{b_iu_{3i-2}, b_iu_{3i}, b_iu_{3i+2}| 1\le i\le k\}, \\
E(C_2,C_3)&=& \{u_iv_{i-1},u_iv_{i+1}|1\le i\le 3k\},\\
E(C_4,C_3)&=& \{a_iv_{3i-2}, a_iv_{3i}, a_iv_{3i+2}| 1\le i\le k\}.
\end{eqnarray*}
where $u_{3k+2}=u_2$, $v_{3k+1}=v_1$ and $v_{3k+2}=v_2.$

Next, we construct a $1$-planar graph $H_k=(V(H_k), E(H_k))$ as illustrated in Fig. \ref{fig3}, where
$$V(H_k)=\bigcup_{i=1}^{4}V(C_i), ~~E(H_k)=\bigcup_{i=1}^4 E(C_i)\bigcup E(C_1,C_2)\bigcup E(C_2,C_3)\bigcup E(C_4,C_3).$$

For any vertex $x_i\in V(C_i)$(where $i=1,2,3,4$), it is evident from the definition of $H_k$ that 
$$|E\big(x_1, V(C_2)\big)|=|E\big(x_4, V(C_3)\big)|=3, ~|E\big(x_2, V(C_1)\big)|=|E\big(x_3, V(C_4)\big)|=1, $$
$$ |E\big(x_2, V(C_3)\big)|= |E\big(x_3, V(C_2)\big)|=2.$$
For convenience, we define $\theta_{i,j}=|E\big(x_i, V(C_j)\big)|$ for $i\neq j$.

%
%
%

\begin{claim}
$\kappa(H_k)=5$.
\end{claim}

\begin{proof}
We  need to proof that any subset  $S\subset V(H_k)$ with $|S|\leq 4$ is not a separator of $H_k$.

Given $S\subset V(H_k)$ with $|S|\leq 4$,  let $t_i=|V(C_i)\cap S|$ for $1\leq i\leq 4$. If $t_i\leq 1$, then $C_i\setminus S$ is either a path or a cycle.
 For integer $1\leq i\leq 4$ with $j=i+1$ or $j=i-1$($j\in\{1,2,3,4\}$), by the definition of $\theta_{i,j}$, we have
\begin{eqnarray}\label{eq1} 
\theta_{i,j}&\geq& |E\big(u, V(C_{j})\setminus S\big)|= |E\big(u, V(C_{j})\big)|- |E\big(u, V(C_{j})\cap S\big)|\geq \theta_{i,j}-t_j. 
\end{eqnarray}  
Furthermore, for $1\leq i\leq 3$ 
 \begin{eqnarray}\label{eq2} 
 |E\big(V(C_i)\setminus S, V(C_{i+1}\big)\setminus S)|&\geq& |E\big(V(C_i), V(C_{i+1})\big)|- |E\big(V(C_i)\cap S, V(C_{i+1})\big)|\nonumber\\ 
&&- |E\big(V(C_i), V(C_{i+1})\cap S\big)|\nonumber\\ 
&\geq& \theta_{i,i+1}|V(C_i)|-\theta_{i,i+1}t_i-\theta_{i+1,i}t_{i+1}. 
\end{eqnarray}   

Now, we consider two cases based on the values of $t_i$($i=1,2,3,4$).

{\bf Case 1}~$t_2+t_3=4$, implying $t_1+t_4=0$. 
It follows that $H_k[V(C_1)\setminus S]=C_1$ and $H_k[V(C_4)\setminus S]=C_4$. 
From equation \ref{eq1}, 
for any vertex $v\in V(C_2)$(or $v\in C_3$), we find $|E\big(v, V(C_1)\setminus S\big)|=1$(or $|E\big(v, V(C_4)\setminus S\big)|=1$).
Consequently, the induced subgraph $H_k[\big(V(C_1)\cup V(C_2)\big)\setminus S]$ and  $H_k[\big(V(C_3)\cup V(C_4)\big)\setminus S]$ are both connected. 
Additionally,  from equation \ref{eq2}, 
we obtain  
$$|E\big(V(C_2)\setminus S, V(C_{3})\setminus S\big)|\geq 6k-2t_2-2t_{3}=6k-8\geq 0.$$
Thus, there is at least one edge connecting 
$H_k[\big(V(C_1)\cup V(C_2)\big)\setminus S]$ and $H_k[\big(V(C_3)\cup V(C_4)\big)\setminus S]$, confirming that $S$ is not a separator of $H_k$.

{\bf Case 2}~$t_2+t_3\leq 3$, which means 
$t_1+t_4\geq 1$. Here, there exists $j\in \{2,3\}$ such that $t_j\leq 1$, leading to $H_k[V(C_j)\setminus S]$ being a path or a cycle. Additionally, from equation \ref{eq1},
for any vertex $u\in V(C_{5-j})$(with $j\in \{2,3\}$), we find $|E\big(u, V(C_{j})\setminus S\big)|\geq 1$.Therefore,   $H_k[\big(V(C_2)\cup V(C_3)\big)\setminus S]$ remains connected. 

When $t_1+t_4=1$, both $H_k[C_1\setminus S] $ and $H_k[C_4\setminus S] $ form a path or a cycle. From equation \ref{eq2}, we have
$$|E\big(V(C_1)\setminus S, V(C_2)\setminus S)|\geq 3k-3t_1-t_2>0,~~|E\big(V(C_4)\setminus S, V(C_3)\setminus S\big)|\geq 3k-3t_4-t_3>0.$$ 
Thus, both $H_k[C_1\setminus S] $ and $H_k[C_4\setminus S] $ connect to $H_k[\big(V(C_2)\cup V(C_3)\big)\setminus S]$, indicating $S$ is not a separator of  $H_k$.
 
 When $t_1+t_4\geq 2$,  then $t_2+t_3\leq 2$. For any vertex  $u_1\in V(C_1)\setminus S$ and $u_4\in V(C_4)\setminus S$, according to equation \ref{eq1},  
  $$|E\big(u_1, V(C_{2})\setminus S\big)|\geq 3-t_2\geq 1,~~|E\big(u_4, V(C_{3})\setminus S\big)|\geq 3-t_3\geq 1.$$ 
  Thus, every vertex in $\big(V(C_1)\cup V(C_4)\big)\setminus S$ is connected to vertices in 
$H_k[\big(V(C_2)\cup V(C_3)\big)\setminus S]$.  Therefore, $S$ is not a separator of $H_k$.
\end{proof}

Note that the graph $H_k$ is bipartite because it is colored in black and white, as shown in Figure \ref{fig3}.
We obtained a graph $G_k$ from $H_k$ by adding a new black vertex  and connecting it to five white vertices in  $C_1$ of $H_k$. Note that $G_k$ is the final graph we intend to construct.

 Clearly, $G_k$ is 1-planar. Moreover, based on the above construction, we have the following claim.

\begin{claim}\label{cla-2}
$G_k$  is bipartite.
\end{claim}

Consequently, by Claim \ref{cla-2} and Lemma \ref{lem-1}(i), $G_k$ is perfect.

\begin{claim}
$\kappa(G_k)=5$.
\end{claim}
\begin{proof}
Since $\kappa(H_k)=5$, by Lemma \ref{lem-4} and considering that $k\ge 10$, it follows that $\kappa(G_k)\ge 5$.  Additionally, since $G_k$ has vertices with degree 5, we conclude $\kappa(G_k)= 5$.
\end{proof}

Since $G_k$ is defined to have  $|V(G_k)|=8k+1$, which is odd, it follows from Claim \ref{cla-2} and Lemma \ref{lem-3} that $G_k$ is non-Hamiltonian.
\end{proof}
\begin{figure}
  \centering

\begin{tikzpicture}[scale=0.48, 
    whitenode/.style={circle, draw, fill=white, inner sep=0pt, minimum size=5pt},  
    blacknode/.style={circle, draw, fill=black, inner sep=0pt, minimum size=5pt},  
    blueedge/.style={draw, blue}, 
    rededge/.style={draw,thick, red} 
]  
		\node [style=whitenode] (0) at (4, 1) {};
		\node [style=blacknode] (1) at (3.89074, 2.03956) {};
		\node [style=whitenode] (2) at (3.56773, 3.03368) {};
		\node [style=blacknode] (3) at (3.04508, 3.93893) {};
		\node [style=whitenode] (4) at (2.34565, 4.71572) {};
		\node [style=blacknode] (5) at (1.5, 5.33013) {};
		\node [style=whitenode] (6) at (0.54508, 5.75528) {};
		\node [style=blacknode] (7) at (-0.477358, 5.97261) {};
		\node [style=whitenode] (8) at (-1.52264, 5.97261) {};
		\node [style=blacknode] (9) at (-2.54508, 5.75528) {};
		\node [style=whitenode] (10) at (-3.5, 5.33013) {};
		\node [style=blacknode] (11) at (-4.34565, 4.71572) {};
		\node [style=whitenode] (12) at (-5.04508, 3.93893) {};
		\node [style=blacknode] (13) at (-5.56773, 3.03368) {};
		\node [style=whitenode] (14) at (-5.89074, 2.03956) {};
		\node [style=blacknode] (15) at (-6, 1) {};
		\node [style=whitenode] (16) at (-5.89074, -0.03956) {};
		\node [style=blacknode] (17) at (-5.56773, -1.03368) {};
		\node [style=whitenode] (18) at (-5.04508, -1.93893) {};
		\node [style=blacknode] (19) at (-4.34565, -2.71572) {};
		\node [style=whitenode] (20) at (-3.5, -3.33013) {};
		\node [style=blacknode] (21) at (-2.54508, -3.75528) {};
		\node [style=whitenode] (22) at (-1.52264, -3.97261) {};
		\node [style=blacknode] (23) at (-0.477358, -3.97261) {};
		\node [style=whitenode] (24) at (0.54508, -3.75528) {};
		\node [style=blacknode] (25) at (1.5, -3.33013) {};
		\node [style=whitenode] (26) at (2.34565, -2.71572) {};
		\node [style=blacknode] (27) at (3.04508, -1.93893) {};
		\node [style=whitenode] (28) at (3.56773, -1.03368) {};
		\node [style=blacknode] (29) at (3.89074, -0.03956) {};
		\node [style=whitenode] (30) at (2.33333, 1) {};
		\node [style=blacknode] (31) at (2.26049, 1.69304) {};
		\node [style=whitenode] (32) at (2.04515, 2.35579) {};
		\node [style=blacknode] (33) at (1.69672, 2.95928) {};
		\node [style=whitenode] (34) at (1.23044, 3.47715) {};
		\node [style=blacknode] (35) at (0.66667, 3.88675) {};
		\node [style=whitenode] (36) at (0.03006, 4.17019) {};
		\node [style=blacknode] (37) at (-0.651572, 4.31507) {};
		\node [style=whitenode] (38) at (-1.34843, 4.31507) {};
		\node [style=blacknode] (39) at (-2.03006, 4.17019) {};
		\node [style=whitenode] (40) at (-2.66667, 3.88675) {};
		\node [style=blacknode] (41) at (-3.23044, 3.47715) {};
		\node [style=whitenode] (42) at (-3.69672, 2.95928) {};
		\node [style=blacknode] (43) at (-4.04515, 2.35579) {};
		\node [style=whitenode] (44) at (-4.26049, 1.69304) {};
		\node [style=blacknode] (45) at (-4.33333, 1) {};
		\node [style=whitenode] (46) at (-4.26049, 0.306961) {};
		\node [style=blacknode] (47) at (-4.04515, -0.35579) {};
		\node [style=whitenode] (48) at (-3.69672, -0.95928) {};
		\node [style=blacknode] (49) at (-3.23044, -1.47715) {};
		\node [style=whitenode] (50) at (-2.66667, -1.88675) {};
		\node [style=blacknode] (51) at (-2.03006, -2.17019) {};
		\node [style=whitenode] (52) at (-1.34843, -2.31507) {};
		\node [style=blacknode] (53) at (-0.651572, -2.31507) {};
		\node [style=whitenode] (54) at (0.03006, -2.17019) {};
		\node [style=blacknode] (55) at (0.66667, -1.88675) {};
		\node [style=whitenode] (56) at (1.23044, -1.47715) {};
		\node [style=blacknode] (57) at (1.69672, -0.95928) {};
		\node [style=whitenode] (58) at (2.04515, -0.35579) {};
		\node [style=blacknode] (59) at (2.26049, 0.306961) {};
		\node [style=blacknode] (60) at (0.66667, 1) {};
		\node [style=whitenode] (61) at (0.34836, 1.97964) {};
		\node [style=blacknode] (62) at (-0.484972, 2.58509) {};
		\node [style=whitenode] (63) at (-1.51503, 2.58509) {};
		\node [style=blacknode] (64) at (-2.34836, 1.97964) {};
		\node [style=whitenode] (65) at (-2.66667, 1) {};
		\node [style=blacknode] (66) at (-2.34836, 0.020358) {};
		\node [style=whitenode] (67) at (-1.51503, -0.58509) {};
		\node [style=blacknode] (68) at (-0.484972, -0.58509) {};
		\node [style=whitenode] (69) at (0.34836, 0.020358) {};
		\node (C1) at (-0.5, 0.020358) {${\color{pink!80} C_1}$};
        \node (C2) at (3, 0.020358) {${\color{pink!80} C_2}$};
       \node (C3) at (5, 0.020358) {${\color{pink!80} C_3}$};
        \node (C4) at (10, 0.020358) {${\color{pink!80} C_4}$};
		\node [style=blacknode] (70) at (9, 1) {};
		\node [style=whitenode] (71) at (7.09017, 6.87785) {};
		\node [style=blacknode] (72) at (2.09017, 10.5106) {};
		\node [style=whitenode] (73) at (-4.09017, 10.5106) {};
		\node [style=blacknode] (74) at (-9.09017, 6.87785) {};
		\node [style=whitenode] (75) at (-11, 1) {};
		\node [style=blacknode] (76) at (-9.09017, -4.87785) {};
		\node [style=whitenode] (77) at (-4.09017, -8.51057) {};
		\node [style=blacknode] (78) at (2.09017, -8.51057) {};
		\node [style=whitenode] (79) at (7.09017, -4.87785) {};
		\node  (82) at (-0.75, -1) {$\cdots$};
		\node  (84) at (-0.75, -9) {$\cdots$};
		\node  (85) at (-0.4, -3.75) {$\cdots$};
		\node (86) at (-0.5, -3) {$\cdots$};

		\draw [style=blueedge] (0) to (1);
		\draw [style=blueedge] (1) to (2);
		\draw [style=blueedge] (2) to (3);
		\draw [style=blueedge] (3) to (4);
		\draw [style=blueedge] (4) to (5);
		\draw [style=blueedge] (5) to (6);
		\draw [style=blueedge] (6) to (7);
		\draw [style=blueedge] (7) to (8);
		\draw [style=blueedge] (8) to (9);
		\draw [style=blueedge] (9) to (10);
		\draw [style=blueedge] (10) to (11);
		\draw [style=blueedge] (11) to (12);
		\draw [style=blueedge] (12) to (13);
		\draw [style=blueedge] (13) to (14);
		\draw [style=blueedge] (14) to (15);
		\draw [style=blueedge] (15) to (16);
		\draw [style=blueedge] (16) to (17);
		\draw [style=blueedge] (17) to (18);
		\draw [style=blueedge] (18) to (19);
		\draw [style=blueedge] (19) to (20);
		\draw [style=blueedge] (20) to (21);
		\draw [style=blueedge] (21) to (22);
		\draw [style=blueedge] (22) to (23);
		\draw [style=blueedge] (23) to (24);
		\draw [style=blueedge] (24) to (25);
		\draw [style=blueedge] (25) to (26);
		\draw [style=blueedge] (26) to (27);
		\draw [style=blueedge] (27) to (28);
		\draw [style=blueedge] (28) to (29);
		\draw [style=blueedge] (29) to (0);
		\draw [style=blueedge] (30) to (31);
		\draw [style=blueedge] (31) to (32);
		\draw [style=blueedge] (32) to (33);
		\draw [style=blueedge] (33) to (34);
		\draw [style=blueedge] (34) to (35);
		\draw [style=blueedge] (35) to (36);
		\draw [style=blueedge] (36) to (37);
		\draw [style=blueedge] (37) to (38);
		\draw [style=blueedge] (38) to (39);
		\draw [style=blueedge] (39) to (40);
		\draw [style=blueedge] (40) to (41);
		\draw [style=blueedge] (41) to (42);
		\draw [style=blueedge] (42) to (43);
		\draw [style=blueedge] (43) to (44);
		\draw [style=blueedge] (44) to (45);
		\draw [style=blueedge] (45) to (46);
		\draw [style=blueedge] (46) to (47);
		\draw [style=blueedge] (47) to (48);
		\draw [style=blueedge] (48) to (49);
		\draw [style=blueedge] (49) to (50);
		\draw [style=blueedge] (50) to (51);
		\draw [style=blueedge] (51) to (52);
		\draw [style=blueedge] (52) to (53);
		\draw [style=blueedge] (53) to (54);
		\draw [style=blueedge] (54) to (55);
		\draw [style=blueedge] (55) to (56);
		\draw [style=blueedge] (56) to (57);
		\draw [style=blueedge] (57) to (58);
		\draw [style=blueedge] (58) to (59);
		\draw [style=blueedge] (59) to (30);
		\draw [style=blueedge] (0) to (59);
		\draw [style=blueedge] (0) to (31);
		\draw [style=blueedge] (1) to (30);
		\draw [style=blueedge] (1) to (32);
		\draw [style=blueedge] (2) to (31);
		\draw [style=blueedge] (2) to (33);
		\draw [style=blueedge] (3) to (32);
		\draw [style=blueedge] (3) to (34);
		\draw [style=blueedge] (4) to (33);
		\draw [style=blueedge] (4) to (35);
		\draw [style=blueedge] (5) to (34);
		\draw [style=blueedge] (5) to (36);
		\draw [style=blueedge] (6) to (35);
		\draw [style=blueedge] (6) to (37);
		\draw [style=blueedge] (7) to (36);
		\draw [style=blueedge] (7) to (38);
		\draw [style=blueedge] (8) to (37);
		\draw [style=blueedge] (8) to (39);
		\draw [style=blueedge] (9) to (38);
		\draw [style=blueedge] (9) to (40);
		\draw [style=blueedge] (10) to (39);
		\draw [style=blueedge] (10) to (41);
		\draw [style=blueedge] (11) to (40);
		\draw [style=blueedge] (11) to (42);
		\draw [style=blueedge] (12) to (41);
		\draw [style=blueedge] (12) to (43);
		\draw [style=blueedge] (13) to (42);
		\draw [style=blueedge] (13) to (44);
		\draw [style=blueedge] (14) to (43);
		\draw [style=blueedge] (14) to (45);
		\draw [style=blueedge] (15) to (44);
		\draw [style=blueedge] (15) to (46);
		\draw [style=blueedge] (16) to (45);
		\draw [style=blueedge] (16) to (47);
		\draw [style=blueedge] (17) to (46);
		\draw [style=blueedge] (17) to (48);
		\draw [style=blueedge] (18) to (47);
		\draw [style=blueedge] (18) to (49);
		\draw [style=blueedge] (19) to (48);
		\draw [style=blueedge] (19) to (50);
		\draw [style=blueedge] (20) to (49);
		\draw [style=blueedge] (20) to (51);
		\draw [style=blueedge] (21) to (50);
		\draw [style=blueedge] (21) to (52);
		\draw [style=blueedge] (22) to (51);
		\draw [style=blueedge] (22) to (53);
		\draw [style=blueedge] (23) to (52);
		\draw [style=blueedge] (23) to (54);
		\draw [style=blueedge] (24) to (53);
		\draw [style=blueedge] (24) to (55);
		\draw [style=blueedge] (25) to (54);
		\draw [style=blueedge] (25) to (56);
		\draw [style=blueedge] (26) to (55);
		\draw [style=blueedge] (26) to (57);
		\draw [style=blueedge] (27) to (56);
		\draw [style=blueedge] (27) to (58);
		\draw [style=blueedge] (28) to (57);
		\draw [style=blueedge] (28) to (59);
		\draw [style=blueedge] (29) to (58);
		\draw [style=blueedge] (29) to (30);
		\draw [style=blueedge] (60) to (61);
		\draw [style=blueedge] (61) to (62);
		\draw [style=blueedge] (62) to (63);
		\draw [style=blueedge] (63) to (64);
		\draw [style=blueedge] (64) to (65);
		\draw [style=blueedge] (65) to (66);
		\draw [style=blueedge] (66) to (67);
		\draw [style=blueedge] (67) to (68);
		\draw [style=blueedge] (68) to (69);
		\draw [style=blueedge] (69) to (60);
		\draw [style=blueedge] (60) to (30);
		\draw [style=blueedge] (60) to (32);
		\draw [style=blueedge] (60) to (34);
		\draw [style=blueedge] (61) to (33);
		\draw [style=blueedge] (61) to (35);
		\draw [style=blueedge] (61) to (37);
		\draw [style=blueedge] (62) to (36);
		\draw [style=blueedge] (62) to (38);
		\draw [style=blueedge] (62) to (40);
		\draw [style=blueedge] (63) to (39);
		\draw [style=blueedge] (63) to (41);
		\draw [style=blueedge] (63) to (43);
		\draw [style=blueedge] (64) to (42);
		\draw [style=blueedge] (64) to (44);
		\draw [style=blueedge] (64) to (46);
		\draw [style=blueedge] (65) to (45);
		\draw [style=blueedge] (65) to (47);
		\draw [style=blueedge] (65) to (49);
		\draw [style=blueedge] (66) to (48);
		\draw [style=blueedge] (66) to (50);
		\draw [style=blueedge] (66) to (52);
		\draw [style=blueedge] (67) to (51);
		\draw [style=blueedge] (67) to (53);
		\draw [style=blueedge] (67) to (55);
		\draw [style=blueedge] (68) to (54);
		\draw [style=blueedge] (68) to (56);
		\draw [style=blueedge] (68) to (58);
		\draw [style=blueedge] (69) to (57);
		\draw [style=blueedge] (69) to (59);
		\draw [style=blueedge] (69) to (31);
		\draw [style=blueedge] (70) to (71);
		\draw [style=blueedge] (71) to (72);
		\draw [style=blueedge] (72) to (73);
		\draw [style=blueedge] (73) to (74);
		\draw [style=blueedge] (74) to (75);
		\draw [style=blueedge] (75) to (76);
		\draw [style=blueedge] (76) to (77);
		\draw [style=blueedge] (77) to (78);
		\draw [style=blueedge] (78) to (79);
		\draw [style=blueedge] (79) to (70);
		\draw [style=blueedge] (70) to (0);
		\draw [style=blueedge] (70) to (2);
		\draw [style=blueedge] (70) to (4);
		\draw [style=blueedge] (71) to (3);
		\draw [style=blueedge] (71) to (5);
		\draw [style=blueedge] (71) to (7);
		\draw [style=blueedge] (72) to (6);
		\draw [style=blueedge] (72) to (8);
		\draw [style=blueedge] (72) to (10);
		\draw [style=blueedge] (73) to (9);
		\draw [style=blueedge] (73) to (11);
		\draw [style=blueedge] (73) to (13);
		\draw [style=blueedge] (74) to (12);
		\draw [style=blueedge] (74) to (14);
		\draw [style=blueedge] (74) to (16);
		\draw [style=blueedge] (75) to (15);
		\draw [style=blueedge] (75) to (17);
		\draw [style=blueedge] (75) to (19);
		\draw [style=blueedge] (76) to (18);
		\draw [style=blueedge] (76) to (20);
		\draw [style=blueedge] (76) to (22);
		\draw [style=blueedge] (77) to (21);
		\draw [style=blueedge] (77) to (23);
		\draw [style=blueedge] (77) to (25);
		\draw [style=blueedge] (78) to (24);
		\draw [style=blueedge] (78) to (26);
		\draw [style=blueedge] (78) to (28);
		\draw [style=blueedge] (79) to (27);
		\draw [style=blueedge] (79) to (29);
		\draw [style=blueedge] (79) to (1);

\end{tikzpicture}
  \caption{A 1-planar drawing of $H_k$}\label{fig3}
\end{figure}
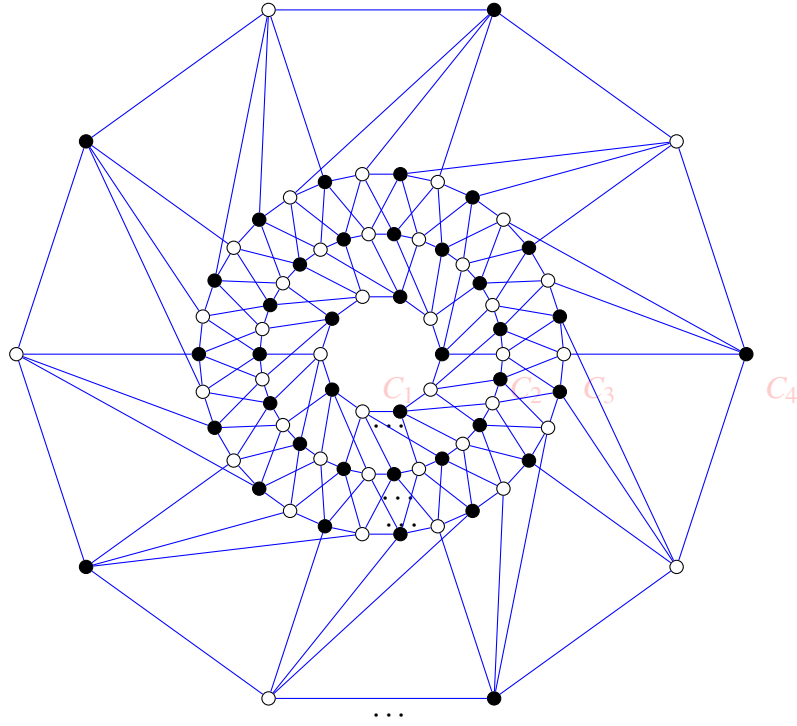

\section{Concluding Remarks}

Theorem \ref{main1} indicates that  connectivity of 5 does not guarantee that 1-planar perfect graphs are Hamiltonian. However, the Hamiltonicity of 6-connected and 7-connected 1-planar perfect graphs remains unknown. Thus,  we propose the following problem with a stronger condition than that of Problem \ref{p1}.

 \begin{prob}
 Is every $6$-connected or $7$-connected $1$-planar perfect graph Hamiltonian?
 \end{prob}

Let $\beta(n)$ represent the maximum number of edges in bipartite $1$-planar graphs of given order $n \geq 4$. Karpov \cite{Karpov} proved that
$$
\beta(n)= \begin{cases}3 n-8 & \text { for even } n \neq 6; \\ 3 n-9 & \text { for odd } n \text { or } n=6.\end{cases}
$$
 This implies that the minimum degree of every $1$-planar bipartite graph is at most 5, resulting in the maximum connectivity  of $5$. Thus, if non-Hamiltonian perfect $1$-planar graphs with connectivity $6$ or $7$ exist, they must be non-bipartite.
 
  Recall that every $1$-planar
graph  with $n\ge 3$ vertices has at most $4n-8$ edges. A simple $1$-planar graph $G$ with $n$ vertices is called \emph{optimal }if it has exactly $4n-8$ edges.
The following theorem is easily derived from existing results.
\begin{thm} 
Let $G$ be an optimal $1$-planar graph with $n$ vertices. Then $G$ is not chordal.
\end{thm}
\begin{proof}
It is known that optimal 1-planar graphs are 4-connected \cite{Suzuki2010}. Zhang et al. \cite{Zhang} showed that any $4$-connected chordal $1$-planar graph with $n$ vertices is a $4$-tree, having exactly $4 n-10$ edges. Therefore, $G$ cannot be chordal, as desired.
\end{proof}

It is intriguing to extend this theorem to the perfection of optimal $1$-planar graphs. Notably, $ K_{2,2,2,2}$ is the unique optimal $1$-planar graph with $8$ vertices, which is perfect by direct verification. However, computer searches reveal  no optimal $1$- planar perfect  graphs with $ 10 \le n \le 20$ vertices. We thus  propose the following conjecture.
\begin{conj}
Every optimal $1$-planar graph with at least $10$ vertices is not perfect.
\end{conj}

\section{Acknowledgment}
The first author would like to thank Misha Lavrov and LeechLattice for their valuable discussions.
We claim that there is no conflict of interest in our paper. No data was used for the research
described in the article.

%
%
%
\end{document}